\documentclass[12pt,leqno]{amsart}
\usepackage{enumerate} % Customize lists
\usepackage{booktabs} %Adjustment of table
\usepackage[utf8]{inputenc} % Required for including letters with accents
\usepackage[T1]{fontenc} % Use 8-bit encoding that has 256 glyphs
\usepackage{amssymb,amsthm}
\usepackage{mathrsfs}
\usepackage{textcomp}
\usepackage{bm}
\usepackage{graphicx}
\usepackage{xcolor}

\newtheorem{thm}{Theorem}%[section]

\newtheorem{lem}{Lemma}
\newtheorem{cor}{Corollary}

\theoremstyle{definition}

\theoremstyle{remark}

\numberwithin{equation}{section}

\usepackage[font=footnotesize,tableposition=top,figureposition=bottom,skip=0pt]{caption}
\usepackage{hyperref}
\hypersetup{%
	linkcolor=blue,
	anchorcolor=black,
	citecolor=red,
	filecolor=magenta,
	menucolor=red,
	urlcolor=magenta,
	colorlinks=true,
	bookmarksnumbered=true,
	bookmarksopen=true,
	bookmarksopenlevel=0,
	breaklinks=true
}

\newcommand{\R}{\mathbb{R}}  % The real numbers.
\begin{document}

%% The title of the paper goes here.  Edit to your title
\title[Weighted Birkhoff ergodic theorem]{Weighted Birkhoff ergodic theorem with oscillating weights }
%% Now edit the following to give your name and address:

\author{Ai-hua Fan}
\address{LAMFA, UMR 7352 CNRS, University of Picardie, 33 rue Saint Leu,80039 Amiens, France}
\email{ai-hua.fan@u-picardie.fr}
%	\urladdr{www.math.sc.edu/$\sim$howard}
\thanks{}
\maketitle

%% If there is another author uncomment and edit the following.

%\author{Second Author}
%\address{Department}
%\email{second@gmail.com}
%\urladdr{http/$\sim$second}

%% If there are three of more authors they are added in the obvious
%% way.

%\subjclass[2000]{Primary }
%    The 2010 edition of the Mathematics Subject Classification is
%    now available.  If you are citing a classification from the
%    new scheme, use the following input coding instead.
%\subjclass[2010]{Primary }

%%% The following is for the abstract.  The abstract is optional and
%%% if not used just delete, or comment out, the following.

\begin{abstract} We consider sequences of Davenport type or Gelfond type and prove that sequences of Davenport exponent larger than $\frac{1}{2}$ are good sequences of weights for   the ergodic theorem, and that the ergodic sums weighted by a sequence of
strong Gelfond property is well controlled almost everywhere. We  prove that for any $q$-multiplicative sequence, the Gelfond property implies the strong Gelfond property and that sequences realized by 
dynamical systems can be fully oscillating and have the Gelfond property. 

%.\cite{bassily2009note},\cite{delange1972fonctions},\cite{grabner1993completely},\cite{liardet1987regularities}
\end{abstract}

%%  LaTeX will not make the title for the paper unless told to do so.
%%  This is done by uncommenting the following.

%% LaTeX can automatically make a table of contents.  This is done by
%% uncommenting the following:

\tableofcontents

%%
%%  To enter text is easy.  Just type it.  A blank line starts a new
%%  paragraph.
%%

\section{Introduction}
A sequence of complex numbers $(w_n)_{n \ge 0} \subset \mathbb{C}$ will be considered
as weights for Birkhoff averages. Oscillating weights of higher order were defined in \cite{F} and oscillating weights of order $1$ appeared earlier in \cite{FJ}. 
Recall that the sequence $(w_n)$ is defined to be {\em oscillating of order $d$} ($d \ge  1$) if for any real polynomial $P$ of degree less than or equal to d we have
$$
    \lim_{N\to \infty} \frac{1}{N}\sum_{n=0}^{N-1} w_n e^{ 2\pi i P(n)} =0.
$$
A {\em fully oscillating sequence} is defined to be an oscillating sequence of all orders. It was proved in \cite{FJ} that if $(w_n)$ is oscillating (i.e. oscillating of order $1$), then the following weighted ergodic limit
\begin{equation} \label{SarnakConj}
   \lim_{N\to \infty} \frac{1}{N}\sum_{n=0}^{N-1} w_n f(T^n x) =0
\end{equation}
exists for any $f \in C(X)$ and any $x \in X$, where $(X,T)$ is some topological dynamical system of zero entropy, for example, an arbitrary homeomorphism of the circle \cite{FJ}. This is related to Sarnak's conjecture \cite{Sarnak, Sarnak2}.
  The following result on everywhere convergence of multiple ergodic limit was proved in \cite{F2}  under the  assumption of full oscillation (only a higher order oscillation suffices for certain dynamics).
  
  \begin{thm} [\cite{F2}] \label{MET} Let $\ell \ge  1$ be an integer.
  Suppose that $(w_n) $ is fully oscillating and that $Tx = Ax +b$ is an affine linear map of zero entropy on a compact abelian group $X$. Then
  the following weighted multiple ergodic limit exists:
\begin{equation}\label{WMEL}
    \lim_{N\to \infty} \frac{1}{N}\sum_{n=0}^{N-1} w_n f_1(T^{q_1(n)} x) \cdots  f_\ell(T^{q_\ell(n)} x) =0,
\end{equation}
for any $f_1, \cdots, f_\ell \in C(X)$,  any polynomial $q_1, \cdots, q_\ell  \in \mathbb{Z}[z]$  such that $q_j(\mathbb{N})\subset \mathbb{N}$
($1\le j\le \ell $) and any $x \in X$.
\end{thm}

The same result in Theorem \ref{MET} was proved in \cite{F} for topological dynamics systems of quasi-discrete spectrum in the sense of Hahn-Parry \cite{HP0}. 

The above discussion is related to Sarnak's conjecture \cite{Sarnak,Sarnak2}, concerning the everywhere convergence of weighted ergodic averages for topological dynamical systems of zero entropy. In this paper we discuss measure-preserving dynamical systems. We shall introduce some other properties of weights which are stronger than oscillation of order one, but different from oscillations of higher order and we shall discuss almost everywhere 
convergence of weighted ergodic averages.
\medskip

For a measure-preserving dynamical system $(X, \mathcal{B}, \nu, T)$
and an integrable function $f \in L^1(\nu)$, we shall study the weighted
ergodic limit
\begin{equation}\label{1}
\lim_{N\to \infty} \frac{1}{N}\sum_{n=0}^{N-1} w_n f(T^n x).
\end{equation}
%where $(X,T)$ is a topological dynamical system of quasi-discrete spectrum in the sense of Hahn-Parry \cite{HP}. 
One of possible assumptions that we shall make on the weights $(w_n)$
is as follows. There exists a constant $h > 0$ such that
\begin{equation}\label{2}
\max_{0\le t \le 1}\left|\sum_{n=0}^{N-1}  w_n e^{2 \pi i nt}\right| = O\left(\frac{N}{\log^h N}\right).
\end{equation}
Then we say $(w_n)$ is of {\em Davenport type}. The largest h will be denoted by
H and called the {\em Davenport exponent} of $(w_n)$. Davenport \cite{Davenport} proved
that the M\"{o}bius sequence $\mu$ has its Davenport exponent $H = +\infty$. Recall
that $\mu(1) = 1$ and $\mu(n) = (-1)^k$ if $n$ is square-free and has $k$ prime
factors. We shall remark that as a consequence of a well-known Davenport-
Erd\"{o}s-LeVeque theorem (see Theorem \ref{DELthm}), the limit (\ref{1}) exists and is equal to zero
almost everywhere under the assumption $H > \frac{1}{2}$ (Theorem \ref{WET}).

If the following stronger assumption is satisfied: there exists a constant
$\frac{1}{2}\le \alpha <1$ such that
\begin{equation}\label{Gproperty}
   \max_{0\le t\le 1} \left|\sum_{n=0}^{N-1} w_n e^{2\pi i nt}\right| =O(N^\alpha),
\end{equation}
we say $(w_n)$ is of {\em Gelfond type}. The smallest $\alpha$ will be denoted by $\Delta$
and called the {\em Gelfond exponent} of $(w_n)$. The Thue-Morse sequence
$(t_n)$ is defined by $t_n = (-1)^{s_2(n)}$ where $s_2(n)$ is the sum of dyadic digits
of $n$. Gelfond \cite{Gelfond} proved that $\Delta =\frac{\log 3}{
\log 4}$ for the Thue-Morse sequence.
The Thue-Morse sequence is a $2$-multiplicative arithmetic function (see Section \ref{q-seq} for the definition
of $q$-multiplicative sequence). 
If $(w_n)$ is a $q$-multiplicative sequence ($q \ge 2$ be an integer) and satisfies the Gelfond property
(\ref{Gproperty}), we shall prove that it satisfies the following strong Gelfond
property
\begin{equation}\label{SGproperty}
   \forall 0\le M <N, \ \ \ \max_{0\le t\le 1} \left|\sum_{M}^{N-1} w_n e^{2\pi i nt}\right| =O((N-M)^\alpha)
\end{equation}
(Theorem \ref{SGthm}). Under the assumption (\ref{SGproperty}) made on $(w_n)$, we have the
following estimate on the size of weighted Birkhoff sum
\begin{equation}\label{SizeBirkhoff}
  \nu\!-\!a.e. \ \ \  \sum_{n=0}^{N-1} w_n f(T^n x) =O(N^\alpha \log^2 N \log^{1+\delta} \log N).
\end{equation}
This holds for any $f \in L^2(\nu)$ and any $\delta > 0$ (Theorem \ref{thm:2}).
\medskip

We shall also study the oscillating properties of sequences of the form $h(S^n x)$, called realizations of dynamical system.  

Here is the organization of the paper. In Section 2, we prove a weighted Birkhoff theorem under the condition $H>\frac{1}{2}$ on the weights of Davenport type (Theorem \ref{WET}). In Section 3, for weights having the strong Gelfond property,
we prove a upper bound for the size of weighted Birkhoff sums (Theorem \ref{thm:2}). In Section 4, we prove that any $q$-multiplicative sequences having Gelfond property must have the strong Gelfond property
(Theorem \ref{SGthm}).
In Section 5, a random $q$-multiplicative sequence is considered and it is proved that it  has the Gelfond property almost surely (Theorem \ref{R-q-M}).   In Section 6, the fully oscillating property and the Gelfond property are proved
for a class of sequences realized by a measure-preserving dynamical systems (Theorem \ref{D-Gelfond}).   

\section{Birkhoff theorem with Davenport weights}
Let $(w_n) \subset \mathbb{C}$ be a sequence of complex numbers and $(u_n)\subset \mathbb{N}$ be a sequence of integers.   We say that the {\em DEL-condition} is satisfied by $(w_n)$ and $(u_n)$ if 
\begin{equation}\label{WETcondition}
   \sum_{n=1}^\infty \frac{1}{n^3} \max_{0\le t \le 1}\left| \sum_{k=0}^{n-1} w_k e^{2\pi i u_k t}\right|^2 <\infty.
\end{equation}
Here DEL refers to Davenport-Erd\"{o}s-LeVeque. The reason for this terminology will soon become clear.  
We say  that $(u_n)$ is a {\em $r$-Bourgain sequence}, $1\le r <\infty$ if for any 
dynamical system $(X, \mathcal{B}, \nu, T)$ the following maximal inequality holds:
\begin{equation}\label{r-B}
     \| \max_{N\ge 1} |A_N f| \|_r \le C  \|f\|_r \quad (\forall f \in L^r(\nu))
\end{equation}
where $A_N f = N^{-1} \sum_{n=0}^{N-1} f\circ T^{u_n}$ and $C>1$ is a constant.

We deduce the following weighted Birkhoff theorem  from a
well-known Davenport-Erd\"{o}s-LeVeque theorem (see Lemma \ref{DELthm}).

\begin{thm}\label{WET}
Let $(w_n)_{n\ge 0} \in \ell^\infty$ and $(u_n)\subset \mathbb{N}$. Suppose \\
\indent \ {\rm (H1)} \ $(u_n)$ is a $r$-Bourgain 
sequence for some $1\le r <\infty$;\\
\indent \ {\rm (H2)} \  $(w_n)$ and $(u_n)$ satisfy  the DEL-condition.\\
%\begin{equation}\label{WETcondition}
%   \sum_{n=1}^\infty \frac{1}{n^3} \max_{0\le t \le 1}\left| \sum_{k=0}^{n-1} w_k e^{2\pi i k t}\right|^2 <\infty.
%\end{equation}
Let $(X, \mathcal{B}, \nu, T)$ be a measure-preserving dynamical system. Then for any $f \in  L^r(\nu)$, the limit
\begin{equation}
\lim_{N\to \infty} \frac{1}{N}\sum_{n=0}^{N-1} w_n f(T^{u_n} x)=0
\end{equation}
holds for $\nu$-almost every $x \in X$ and in
$L^r$-norm.
\end{thm}

Davenport,
Erd\"{o}s and LeVeque \cite{DEL}, in their study of uniform distribution based on Weyl's criterion, proved the following theorem. It
was stated for a very special case, but the proof is valid for the general
statement below.

\begin{lem}
[Davenport-Erd\"{o}s-LeVeque \cite{DEL}] \label{DELthm}
 Let $(\xi_n)_{n\ge 0}$ be a sequence
of bounded random variables defined on some probability space,
such that $\|\xi_n\|_\infty = O(1)$ as $n \to \infty$. Suppose the following Davenport-
Erd\"{o}s-LeVeque condition is satisfied
\begin{equation}\label{DELcondition}
   \sum_{n=1}^\infty \frac{\mathbb{E} |\xi_0 + \xi_1 +\cdots + \xi_{n-1}|^2}{n^3}
     <\infty.
\end{equation}
Then almost surely we have
$$
    \lim_{n \to \infty} \frac{1}{n}\sum_{k=0}^{n-1} \xi_k = 0.
$$
\end{lem}

{\it Proof of Theorem~\ref{WET}.} First assume $f \in L^\infty(\nu)$. Let $\xi_n = f(T^{u_n}x)$.
The almost everywhere convergence  follows directly from Davenport-Erd\"{o}s-LeVeque’s theorem,
the hypothesis (H2) and the spectral lemma which gives
$$
\mathbb{E} |\xi_0 + \xi_1 +\cdots + \xi_{n-1}|^2 = \int \left| \sum_{k=0}^{n-1}w_k e^{2\pi i u_k t}\right|^2 d\sigma_f(t)
 \le \max_{0\le t \le 1} \left| \sum_{k=0}^{n-1}w_k e^{2\pi i u_k t}\right|^2,
 $$
where $\sigma_f$ is the spectral measure of $f$ defined on the circle $\mathbb{R}/\mathbb{Z}$.

Now assume $f \in  L^r(\nu)$. For any $\epsilon  > 0$, there exists $g \in  L^\infty(\nu)$ such
that $\| f - g\|_r<\epsilon$. By writing $f = g + (f - g)$, we get
$$
\left|\frac{1}{N} \sum_{n=0}^{N-1} w_n f(T^n x)\right|
\le \left|\frac{1}{N} \sum_{n=0}^{N-1} w_n g(T^n x)\right| + \|w\|_\infty\left|\frac{1}{N} \sum_{n=0}^{N-1} |f-g|\circ T^n(x)\right|.
$$
Applying the above proved result to $g$, %and the classical Birkhoff theorem to $|f-g|$, 
we get
$$
\nu\!-\!a.e. \ \ \ \limsup_{N\to \infty} \left|\frac{1}{N} \sum_{n=0}^{N-1} w_n f(T^n x)\right| \le \max_{N\ge 1}A_N |f-g|(x).
$$
%where $\mathcal{I}_\nu$ is the $\sigma$-field of $T$-invariant sets. 
Then, by the hypothesis (H1),  we have
%Taking the expectation of both sides of the above inequality leads to
$$
 \mathbb{E} \limsup_{N\to \infty}   \left| \frac{1}{N} \sum_{n=0}^{N-1} w_n f(T^n x)\right|^r \le C^r \|f-g\|_r^r<C^r \epsilon^r.
$$
We have thus finished the almost everywhere convergence,  for $\epsilon$ can be arbitrarily small.   

After having proved the pointwise convergence, it is easy to deduce the $L^r$-convergence by Lebesgue's dominated convergence theorem.
$\Box$

\medskip
{\bf Remark 1.} If $(w_n)$ is of Davenport type with Davenport exponent
$H > \frac{1}{2}$, the series in the hypothesis (\ref{WETcondition}) is bounded by
$\sum \frac{1}{n \log^{2H} n}$
which is finite.   On the other hand, the classical ergodic  maximal inequality means that the set of natural numbers is a $1$-Bourgain sequence.  Therefore we have the following corollary.

\begin{cor}\label{Cor1}  Suppose that $(w_n)\subset \ell^\infty$ is of Davenport type with Davenport exponent
$H > \frac{1}{2}$. Let $(X, \mathcal{B}, \nu, T)$ be a measure-preserving dynamical system. Then for any $f \in  L^1(\nu)$, the limit
\begin{equation}
\lim_{N\to \infty} \frac{1}{N}\sum_{n=0}^{N-1} w_n f(T^{u_n} x)=0
\end{equation}
holds for $\nu$-almost every $x \in X$ and in
$L^1$-norm.   
\end{cor}

 In \cite{AKLR} (see Proposition 3.1 there), Abdalaoui et al proved Corollary \ref{Cor1}
 under the stronger condition $H > 1$ using a more elementary argument than that of Davenport-Erd\"{o}s-
LeVeque. Actually, the result in \cite{AKLR}
 was only
stated for M\"{o}bius function, but the argument works for all $(w_n)$
with $H > 1$.
\medskip

Bourgain proved that polynomial sequence  $(u_n):=p(n)$ with $p(x) \in \mathbb{Z}[x]$ is a $r$-Bourgain sequence for $r>1$. So, we have the following corollary.

\begin{cor} \label{Cor2}  
Let $(X, \mathcal{B}, \nu, T)$ be a measure-preserving dynamical system and let $p(x) \in \mathbb{Z}[x]$ taking values in $\mathbb{N}$ and $(w_n)\in \ell^\infty $.  Assume $r>1$. Suppose that 
$$
         \max_{0\le t\le 1 } \left|\sum_{n=0}^{N-1} w_n e^{2\pi ip(n) t}\right| \le C\frac{N}{\log^h N}
$$
with $h>\frac{1}{2}$ and $C>1$.
Then for any $f \in  L^r(\nu)$, the limit
\begin{equation}
\lim_{N\to \infty} \frac{1}{N}\sum_{n=0}^{N-1} w_n f(T^{u_n} x)=0
\end{equation}
holds for $\nu$-almost every $x \in X$ and in
$L^r$-norm.
\end{cor}

Eisner \cite{Eisner}  stated  Corollary \ref{Cor2} with
 $(w_n):=(\mu(n))$ the  M\"{o}bius function  and observed that it can be proved 
using the argument in \cite{AKLR}.

{\bf Remark 2.}  Assume $f \in L^2(\nu)$ and $\mathbb{E}f = 0$. Cohen and Lin \cite{CL}
proved that the condition
\begin{equation}\label{CLcondition}
\sum_{n=1}^\infty \frac{\log n}{n^3} 
\|f + f \circ  T + \cdots  + f \circ T^{n-1} \|_2^2 <\infty
\end{equation}
is necessary and sufficient for the ergodic Hilbert transform
$\sum_{n=1}^\infty \frac{f(T^n x)}{n}$
to converge in $L^2$-norm. Cuny \cite{Cuny} proved that this condition implies
the almost everywhere convergence of
$\sum_{n=1}^\infty \frac{f(T^n x)}{n}$, a fortiori the almost
everywhere convergence to zero of $\frac{1}{n}\sum_{k=0}^{n-1} f(T^kx)$. The Cohen-Lin's
condition (\ref{CLcondition}) can be stated as follows
$$ %\begin{equation}\label{CLcondition}
\sum_{n=1}^\infty \frac{\log n}{n^3}
\left\|\sum_{k=0}^{n-1}  e^{2 \pi i k t} \right\|_{L^2(\sigma_f)}^2 <\infty
$$ %\end{equation}
which is to be compared with (\ref{WETcondition}). The almost everywhere convergence of ergodic series 
of the form $\sum a_n f(T^n x)$ is studied in \cite{F0} and \cite{CF}. See the references therein on the subject.  
\medskip

{\bf Remark 3.} If the limit (\ref{1}) exists almost everywhere and in $L^1$-
norm for every dynamical system $(X, \mathcal{B}, \nu, T)$ and for every $f \in L^1(\nu)$,
we call $(w_n)$ a good sequence of weights for the ergodic theorem. The
conclusion of Theorem \ref{WET} is a little bit stronger because the limit is always
zero. Lesigne and Mauduit \cite{LM} (p.151-152) proved that, given a $q$-
multiplicative sequence $(w_n)$, a necessary and sufficient condition for all
sequence $(\rho(w_n))$, $\rho$ being a continuous function defined on the circle,
to be good is that for every real number $t$ and every integer $d\ge 1$ the
following limit exists
\begin{equation}\label{LMcondition}
     \lim_{N\to \infty} \frac{1}{N} \sum_{n=0}^{N-1} w_n^d e^{2 \pi i n t}.
\end{equation}
That is the case when $(w_n)$ is a $q$-multiplicative sequence taking  only a finite number of values.
\medskip

{\bf Remark 4.} In \cite{LMM}, Lesigne, Mauduit and Moss\'e associated a skew
product to a $q$-multiplicative sequence. They proved that if the cocycle
is weakly mixing, all continuous images of every orbits of the skew
product dynamics are good.
\medskip

%{\bf Remark 5.}  Properties similar to (\ref{LMcondition}) were used in \cite{F} and \cite{FJ} to
%study the everywhere convergence of weighted Birkhohh averages, related to Sarnak’s conjecture.

\section{Size of weighted Birkhoff sums\label{sec:Weight-sums}}
Let $ (w_{n})_{n\geq 0} $ be a sequence of complex numbers and $ (u_{n})_{n\geq 0} $ be a strictly increasing sequence of positive integers. For any integer $ N \geq 1$, denote
\[
S_{N}^{w,u}(x) = \sum_{n=0}^{N-1}w_{n}\mathrm{e}^{2\pi i u_{n} x}.
\]
The following condition will be considered
\begin{equation}\label{eq:condition-1}
\max\limits_{x\in\R}\lvert S_{N}^{w,u}(x)\rvert \leq CN^{\beta} \ \ \ (\forall\, N\geq 1)
\end{equation}
for some $ 1/2\le \beta <1 $ and $ C > 0 $. Sometimes, the following stronger condition will also be considered
\begin{equation}\label{eq:condition-2}
\max\limits_{x\in\R}\lvert S_{N}^{w,u}-S_{M}^{w,u} \rvert \leq D(N-M)^{\beta}\ \ \ (\forall\, N\geq M\geq 1)
\end{equation}
for some $ 1/2\le \beta <1 $ and $ D > 0 $.

Give a measure-preserving dynamical system $ (X,\mathscr{B},\nu, T) $.
 %where $ (X,\mathscr{B}, \mu) $ is a probability space and $ T\colon X\to X $ is a measurable map and $ \mu $ is $ T $-invariant.
 We shall consider, for $ f\in L^{2}(\mu) $, the weighted Birkhoff sums
\[
\sigma_{N}^{w,u}f(x) := \sum_{n=0}^{N-1}w_{n}f(T^{u_n}x).
\]
\begin{thm}\label{thm:2}
	Let $ (w_{n})_{n\geq 0} $ be a sequence of complex numbers and $ (u_{n})_{n\geq 0} $ be a strictly increasing sequence of positive integers. Suppose that $ (w_{n}) $ and $ (u_{n}) $ satisfy the condition \eqref{eq:condition-2} with $  \frac{1}{2}\le \beta <1 $
and that $\phi: \mathbb{R}_+ \to \mathbb{R}_+$ is an increasing function such that $\phi(x) \le C \phi(2x)$
for some $C>1$ and for all $x >0$, and
\begin{equation}\label{phicondition}
    \sum_{n=1}^\infty \frac{1}{n\phi(n)} <\infty.
\end{equation}
 Then for any dynamic system $ (X,\mathcal{B}, \mu, T) $ and any $ f\in L^{2}(\mu) $, for $ \mu $-almost all $ x $, we have
	\[
	 \lim\limits_{N\to\infty} \frac{1}{N^{\beta}\log^{2}N \phi(\log N)}\sum_{n=0}^{N-1}w_{n}f(T^{u_n}x) = 0.
	\]
\end{thm}

\begin{proof}
	For integers $ n\geq m\geq 0 $, define $ g_{m,n} =D^{2}(n-m)^{2\beta} $. 	By  the spectral lemma (see \cite{Krengel}, p. 94), for all $ 0\leq a < b $ we have
	\[
	\int \bigg\lvert \sum_{n=a}^{b'-1}w_{n}f(T^{n}x) \bigg\rvert^{2} \mathrm{d} \mu(x) \leq g_{a,b}.
	\]
	Since $ 2\beta \geq 1 $, we have $ (x+y)^{2\beta} \geq x^{2\beta} + y^{2\beta} $ for all $ x\geq 0 $ and $ y\geq 0 $ so that for $ 0\leq a\leq b \leq c $, we have the super-additivity
	\[
	g_{a,c}\geq g_{a,b}+ g_{b,c}.
	\]
        Then using the condition \eqref{eq:condition-2}, we can
	apply the inequality (2.2) in \cite{Moricz} with $ \gamma =2 $ to get
	\[
 	\int \max\limits_{a\leq b'\leq b} \bigg\lvert \sum_{n=a}^{b'-1}w_{n}f(T^{n}x) \bigg\rvert^{2} \mathrm{d} \mu(x) \leq D^{2}(b-a)^{2\beta}\log^{2}2(b-a).
	\]
	In particular, for a fixed $ \rho > 1 $, we have
	\[
	\forall\, m\geq 0,\quad \lVert M_{m} \rVert_{L^{2}(\mu)} = \mathrm{O}(m\rho^{m\beta})
	\]
	where
	\[
	M_{m}(x) = \max\limits_{\rho^{m}\leq b \leq \rho^{m+1}}\bigg\lvert \sum_{\rho^{m}\leq n <b}w_{n}f(T^{u_n}x) \bigg\rvert.
	\]
	It follows that
	\[
	\bigg\lVert \sum_{m=1}^{\infty}\frac{M_{m}}{m^{2}\phi(m)\rho^{m\beta}}\bigg\rVert_{L^{2}(\mu)}
       = \mathrm{O}\bigg(\sum_{m=1}^{\infty}\frac{1}{m \phi(m)}\bigg) < +\infty.
	\]
	So, for $ \mu $-almost all $ x $, we have
	\[
	M_{m}(x)  = o(m^{2}\phi(m)\rho^{m\beta}).
	\]
	For any $ N \geq 1$, there exists $ m $ such that $ \rho^{m} \leq N < \rho ^{m+1} $.
	So we get
	\begin{align*}
	\bigg\lvert \sum_{n=1}^{N-1} w_{n}f(T^{n}x) \bigg\rvert
\leq \sum_{k=0}^m M_{k}
	=o\bigg(\sum_{k=1}^{m}k^{2} \phi(k) \rho^{k \beta}\bigg)=o(m^2 \phi(m) \rho^{\beta m}).
	\end{align*}
Here we have used the monotonicity of  $\phi$. % the last sum is bounded by the integral
%$$
% \int_1^{m+1} x^2\phi(x) \rho^{\beta x} d x
%  = \frac{1}{\log^3 \rho}\int_\rho^{\rho^{m+1}} \log^2 y \phi(\log_\rho y) y^\beta d y
%$$
%(for the last equality we made the change of variables $y = \rho^x)$.
Since $\phi(2x) \le C \phi(x)$ and $\rho^{m}\le  N$, $m^2 \phi(m) \rho^{\beta m}$ is bounded by
$N^\beta \log^2 N \phi(\log N)$, up to the multiplicative constant.
	We have thus completed the proof.
\end{proof}
A weaker result was obtained in \cite{DS}. 

For q-multiplicative sequences, we shall prove the strong Gelfond
property (\ref{SGproperty}), which will allows us to apply Theorem \ref{thm:2} to $q$-multiplicative sequences.

\section{Strong Gelfond property of $q$-multiplicative sequences}\label{q-seq}

Let $ q \geq 2 $ be an integer. A function $ f\colon \mathbb{N} \rightarrow \mathbb{C}$ is said to be {\em $ q $-multiplicative} if
\begin{equation}
f( aq^{n}+b ) = f( aq^{n} ) f(b)\label{eq:condationof-qmf}
\end{equation}
for all integers $ m \geq 1 $, $ a\geq 0 $ and $ 0 \leq b < q^{m} $.
A $ q $-multiplicative function $ f $ is completely determined by its values $ f( aq^{m} )$
for $ m\geq 0 $ and $ 0 \leq a < q $, which constitute the so-called skeleton of $ f $. Indeed, for $ n = \sum a_{j} q^{j} $ where $ a_{j} \in \{ 0, 1, \cdots, q-1  \} $, we have
\[
f(n) = \prod f( a_{j} q^{j} )  .
\]
It is also clear that $ f(0) = 0 $ or $ f(0) = 1 $. In the following, we assume that $ q $-multiplicative sequences take values
 in $ \mathbb{U} := \{ \xi \in \mathbb{C} \colon \vert \xi \vert = 1  \} $, the multiplicative group of complex numbers of modulus $ 1 $.

For a $ q $-multiplicative function $ f$,   we consider the trigonometric polynomials
\[
S_{N}^{f} (x) := S_{N}(x) := \sum_{n=0}^{N-1}f(n)\,\mathrm{e}^{2\pi i n x} \quad (x \in \mathbb{R}) ,
\]
where $ N \geq 1 $. For simplicity, we will write $ e(x) $ for $ \mathrm{e}^{2\pi i x } $.
% TO DO Check by definition of q-multiplicative function
Notice that for any $ x $, $ n\mapsto e(nx) $ is $ q $-multiplicative for any given $ q $. 
%We also notice that
% TO DO Exercise: Proof the fellowing inequality
%\[
%\max\limits_{x\in \mathbb{R}} \lvert S_{N}^{f} \rvert \geq \sqrt{N}.
%\]
Typical examples of $q$-multiplicative sequences are $e^{2\pi i c s_q(n)}$ where $c$ is a fixed real number and $s_q(n)$ is the sum of digits in the $q$-adic expansion of $n$.
\begin{thm}\label{SGthm}
Let $ f\colon \mathbb{N} \rightarrow \mathbb{U} $ be a $ q $-multiplicative function. Suppose that there exist constants $ C \geq 1 $ and $ \frac{1}{2} \le \beta < 1 $ such that
\[
\max\limits_{x\in \mathbb{R}} \lvert S_{n}^{f}(x) \rvert \leq C n^{\beta} \quad (\forall \, n \geq 1) .
\]
Then there exists a constant $ D > 0$ such that
\[ \max\limits_{x\in \mathbb{R}} \lvert S_{n}^{f}(x) - S_{m}^{f}(x) \rvert \leq D (n-m)^{\beta} \quad(\forall \,  n\geq m \geq 0).\]
\end{thm}
\begin{proof}
We simply write $ S_{n} $ for $ S_{n}^{f} $. For $ a \in \mathbb{N} $ and $ j \in \mathbb{N} $, the $ q $-multiplicity implies that
\begin{align*}
S_{(a+1)q^{j}}(x) - S_{a q^{j}}(x) = \sum_{n = 0 }^{q^{j} - 1} f(a q^{j} +n ) \, e((a q^{j}+ n)x)
= f(a q^{j})\, e(a q^{j})\, S_{q^{j}}(x) .
\end{align*}
Then, by the hypothesis, we get
\begin{equation}
\lvert S_{(a+1) q^{j}}(x) - S_{a q^{j}}(x) \rvert \leq {C} q^{j\beta}.
\label{eq:thm1eq1}
\end{equation}
For any couple $ (n, m) \in \mathbb{N}^{2}$ with $ 0 \leq m < n $, there exists a unique integer $ j \in \mathbb{N} $ such that $ q^{j} \leq n-m < q^{j+1} $. We are going to show, by induction on $ j \geq 0 $, that
\begin{equation}
\lvert S_{n}(x) - S_{m}(x) \rvert \leq Kq^{j \beta}
\label{eq:thm1eq2}
\end{equation}
for all $ (n, m)\in \mathbb{N}^{2} $ such that $ q^j\le n-m < q^{j+1} $, where $ K \geq q^3 $ is a constant to be determined.
Since $K\ge q^3$, \eqref{eq:thm1eq2} is trivially true for $j=0, 1, 2$.

First we prove a weaker version of \eqref{eq:thm1eq2} by induction on $ j \ge 1$: the inequality \eqref{eq:thm1eq2}
 holds for all $(n, m) \in \mathbb{N}^2$ such that $0\le n-m <q^{j+1}$ and that one of $n$ and $m$ is of the form $aq^{j-1}$
 for some $a\in \mathbb{N}$. We have seen that \eqref{eq:thm1eq2} holds for $j=1,2$ (initiation of induction). Now suppose that 
 \eqref{eq:thm1eq2} holds for $j=1, 2, \cdots, k$ (with $k\ge 2$).
 Assume $ 0\leq n-m < q^{k+2}$ and $m=aq^k$ for some $a\in \mathbb{N}$ (the reasoning is the same if $n=aq^k$). Let  $ \tau = b q^{k} \leq n $ be the integer with the largest $ b \in \mathbb{N} $. Then $ 0 \leq n - \tau < q^{k} $. 
 Notice that $\tau= aq^{k-2}$ with $a = b q^2$.  
 %Assume
%\[
%q^{j_{1}} \leq \sigma - m < q^{j_{1}+1}, \quad
% q^{k'} \leq n -\tau <  q^{k'+1}, \ \ \ \ 0\le k' \le k-1.
%\]
%where $ 0\leq j_{1},j_{2} \leq k-1$. By the induction hypothesis, we have
%Remark that $\tau=b'q^{k'-1}$ with $b' = bq^{k-k'+1} \in \mathbb{N}$ when $k'=1$ (no need to check this when $k'=0$). 
By the  hypothesis of induction, we have
\begin{equation}
%\vert S_{\sigma}(x) - S_{m}(x) \vert \leq K q^{(k-1)\beta},\quad
\vert S_{n}(x) - S_{\tau}(x) \vert %\leq K q^{(k'-1)\beta}
\le K q^{(k-1)\beta} .
\label{eq:thm1eq3}
\end{equation}
We now estimate
\begin{align*}
\lvert S_{n}(x) - S_{m}(x) \rvert &\leq \lvert S_{n}(x) - S_{\tau}(x) \rvert %+ \lvert S_{\tau}(x) - S_{\sigma}(x) \rvert	
+ \lvert S_{\tau}(x) - S_{m}(x) \rvert \\
&\leq K q^{(k-1)\beta} + \sum_{\ell=a}^{b-1}\lvert S_{(\ell+1)q^{k}}(x) - S_{\ell q^{k}}(x) \rvert \\
&\leq Kq^{(k-1)\beta} + (b-a)C q^{k\beta},
\end{align*}
where the second inequality follows from \eqref{eq:thm1eq3} and the third one follows from~\eqref{eq:thm1eq1}. 
Remark that
$ (b-a)q^{k} \leq n-m < q^{k+2}$, which implies $ b-a \leq q^{2} $. Therefore, we get
\[
\lvert S_{n}(x) - S_{m}(x) \rvert \leq (Kq^{-2\beta} + C q^{2-\beta}) q^{(k+1)\beta}.
\]
Since $ \beta \ge \frac{1}{2} $, we have $ q^{2\beta} \ge q \geq 2 $. We can take $ K \geq q^3 $ large enough so that $ Kq^{-2\beta} + Cq^{2-\beta} \leq K $. Actually, we can take any $K$ such that
\[
K\ge q^3, \quad K \ge  \frac{Cq^{1-\beta}}{1 - 2q^{-2\beta}} = \frac{Cq^{2+\beta}}{q^{2\beta} - 1}.
\]

Now let us deduce the inequality \eqref{eq:thm1eq2} from its weak version that we have just proved. Assume 
$q^j\le n-m <q^{j+1}$ with $j\ge 2$. Let $aq^{j-1}$ be an integer such that
$$
 \left| \frac{n+m}{2} - aq^{j-1}\right| < q^{j-1}.
$$
Then, writing $n - aq^{j-1} = (n - m)/2 + ((n+m)/2 - aq^{j-1})$, we get
$$
  % q^j \le \frac{1}{2}q^{j+1} -q^{j-1} 
   0\le n - aq^{j-1} <\frac{1}{2} q^{j+1} +q^j <q^{j+1}.
$$
Similarly we have
$$
      0\le aq^{j-1} - m <q^{j+1}.
$$
Then, by the weak version of \eqref{eq:thm1eq2}, we get
$$
   |S_n(x) - S_m(x)| \le |S_n(x) - S_{aq^{j-1}}| + |S_{aq^{j-1}} - S_m(x)| \le 2 K q^{j \beta}.
$$

Finally we can take
$$
  D = 2 K =\max \left\{2q^3, \frac{2C q^{2+\beta}}{q^{2\beta}-1}\right\}.
$$
\end{proof}

The above proof can be simplified a little bit when $q\ge 3$.
\medskip

%In \cite{DS}, 

%For $ q $-multiplicative sequence $ (w_{n}) $ taking value in $ \mathbb{U} $ and for the sequence $ (u_{n}) = \mathbb{N} $, the condition %\eqref{eq:condition-2} can be replaced by the condition \eqref{eq:condition-1}, which is actually equivalent to the condition %\eqref{eq:condition-2}\footnote{Theorem \ref{thm:thm1} in Section \ref{sec:Def-thm}}.

The Thue-Morse sequence $ (t_{n}) $ is defined by $ (-1)^{s_2(n)} $, where $ s_2(n) =\sum a_{i} $ is the sum of dyadic digits of $ n = \sum a_j 2^j$ ($a_j = 0, 1$).
The sequences $(t_n^{(c)})$ defined by $t_n^{(c) = ^{2\pi i c s_2(n)}}$ ($0<c<1$) are also studied in the literature. 
 They are  $ 2 $-multiplicative.  Recall that $t_n = t_n^{(1/2)}$.
 
 A.~O.~Gelfond  \cite{Gelfond} proved
\[
\bigg\lvert \sum_{n=0}^{N-1}t_{n}\mathrm{e}^{2\pi inx}  \bigg\rvert \leq C N^{\frac{\log3}{\log4}} \quad (\forall\, N\geq 1,\ \forall\,x\in\R )
\]
where $ C > 0 $ is a constant.

For a general $ q$-multiplicative sequence $ f(n) $ taking value in $ \mathbb{U} $, we have
\[
\bigg\lvert \sum_{n=0}^{N-1}f(n)\mathrm{e}^{2\pi inx} \bigg\rvert \leq CN^{\beta} \quad(\forall\, N\geq 1\ \forall\, x\in\R)
\]
under the condition
\begin{equation}\label{eq:condition-3}
\overline{\lim\limits_{\scriptsize m\to\infty}}\max\limits_{x\in \R}\lvert F_{m}(x)F_{m+1}(x)\cdots F_{m+r-1}(x)\rvert \leq q^{r\beta}
\tag{3}
\end{equation}
for some $ r \geq 2 $, where
\[
F_{m}(x) = \sum_{j=0}^{q-1}f(jq^{m})\mathrm{e}^{2\pi ijq^{m}x }.
\]
Using this idea with $r=2$, Mauduit, Rivat and S\'ark\"{o}zy \cite{MRS} proved that
\[
\bigg\lvert \sum_{n=0}^{N-1}t_{n}^{(c)}\mathrm{e}^{2\pi inx}  \bigg\rvert \leq C N^{1- \frac{\pi^2}{20\log2} \|c\|^2} \quad (\forall\, N\geq 1,\ \forall\,x\in\R )
\]
where $ \|c\|=\inf_{n \in \mathbb{Z}} |c-n| $. This estimate is not optimal, but it gives an interesting upper bound.

In \cite{Ko}, Konieczny proved that for the  Thue-Morse sequence $t$ and for any positive integer $s$, we have $\|t\|_{U^s[N]} = O(N^{-c}) $ for some $c:=c(s)>0$ where $\|\cdot\|_{U^s[N]}$
denotes the $s$-th Gowers uniformity norm. The Gowers uniform norm would be a good tool to study the properties of a given sequence which are examined in this note.

        \section{Random $q$-multiplicative sequences}
        Here we construct some random $q$-multiplicative sequences and prove that they have almost surely the Gelfond property. 
        A $q$-multiplicative sequence $(f(n))_{n\ge 0}$ is determined by its values 
        $f(0)=1$ and $(f(q^t), \cdots, f((q-1)q^t))$ for $t\ge 0$. Take a sequence of random vectors $X=(X_t)_{t \ge 0}$ where
        $$
           X_t = (X_t^{(1)}, \cdots, X_t^{(q-1)}).
        $$
Then we have a random $q$-multiplicative sequence $f_X(\cdot)$ defined by
$$
   f_X(n) =\prod X_{n_k}^{(\epsilon_k)}
   \ \ \ \ \mbox{\rm for}\ \ \ n =\sum_k \epsilon_k q^{n_k}.
$$

Let 
$$
      S_{X, N}(t) = \sum_{n=0}^{N-1} f_X(n) e^{2\pi i n t}.
$$
Then 
$$
   S_{X, q^n}(t) = \prod_{j=0}^{n-1} (1+X_j^{(1)}e^{2\pi i q^jt}+ \cdots + X_j^{(q-1)}e^{2\pi i (q-1)q^jt} ).
$$
The following theorem gives us  a upper bound for the Gelfond exponent of the random $q$-multiplicative sequence.  
        
\begin{thm}\label{R-q-M}
  Suppose that all the random variables $X_t^{(k)}$ ($t \ge 0, 1\le k\le q-1$) are independent and uniformly distributed
  on the circle, that means, $X_t^{(k)} = e^{2\pi i \omega_t^{(k)}}$  where $\omega_t^{(k)}$ are independent and uniformly distributed
  on $[0,1]$. Then there exists a constant $1/2\le \alpha <1$ such that almost surely
  $$
      \|S_{X, N} (\cdot)\|_\infty = O(N^\alpha)
  $$
\end{thm}        

\begin{proof} 
We follow the idea of Salem and Littlewood developed by Kahane (see \cite{Kahane}).
Let $M_N =  \|S_{X, N} (\cdot)\|_\infty$. Observe that $S_{X, q^n}$ is a trigonometric polynomial 
of degree $q^n$. By Bernstein's inequality, %that the derivative of a trigonometric polynomial is bounded by 
we have 
$$\max_{0\le t\le 1} |S'_{X, q^n}(t)| \le q^n \max_{0\le t\le 1}|S_{X, q^n}(t)|.$$ It follows that 
$|S_{X, q^n}(t)|\ge \frac{1}{2} M_{q^n}$ over an random interval $I$ of length at least $q^{-n}$. Therefore
$$
          M_{q^n}^\lambda \le  2^\lambda \int_I |S_{X, q^n}(t)|^\lambda dt
               \le 2^{\lambda} q^n \int_0^1 |S_{X, q^n}(t)|^\lambda d t.
$$
It follows that
$$
    \mathbb{E} M_{q^n}^\lambda 
               \le 2^{\lambda} q^n \int_0^1 \mathbb{E} |S_{X, q^n}(t)|^\lambda d t = 2^\lambda [q^{1+\lambda} \varphi(\lambda)]^n
$$
where
$$
   \varphi(\lambda) = q^{-\lambda} \int_0^1\cdots \int_0^1 |1 + e^{2\pi i x_1}+\cdots + e^{2\pi i x_{q-1}}| dx_1 \cdots dx_{q-1}.
$$
Fix $\tau>0$. By Markov's inequality, we have
$$
   P(M_{q^n} \ge q^{n\tau}) \le 2^\lambda e^{n[(1+ (1-\tau)\lambda)\log q -\Phi(\lambda)]}
$$ 
where $$
\Phi(\lambda) = -\log \varphi(\lambda).$$
 Notice that $\Phi$ is increasing and concave, and $\Phi(0+)=0$ and $\Phi(+\infty)=+\infty$.
Let $\tau^*$ be such that $\lambda \mapsto [1 + (1-\tau^*)\lambda]\log q$ be a tangent of $\Phi$. In other words  the following system
 has a solution $(\tau^*, \lambda^*)$:
\begin{equation}\label{R1}
      (1+(1-\tau^*)\lambda^*) \log q - \Phi(\lambda^*)=0,
      %\quad (1-\tau^*)\log q = \frac{\varphi'(\lambda^*)}{\varphi(\lambda^*)}
\end{equation}
\begin{equation}\label{R2}
     % (1+(1-\tau^*)\lambda^*) \log q - \Phi(\lambda^*)=0,
      %\quad
       (1-\tau^*)\log q = \frac{\varphi'(\lambda^*)}{\varphi(\lambda^*)}.
\end{equation}
It is clear that $\tau^*<1$. So, for any $\tau \in (\tau^*, 1)$ we have
$$
      P(M_{q^n} \ge q^{n\tau}) \le 2^{\lambda^*} e^{ -\delta n} 
$$
where  $\delta= -[1+(1-\beta)\lambda^*]\log q + \Phi(\lambda^*)>0$. Then by Borel-Cantelli lemma, we have
$$
      a.s. \quad M_{q^n} = O(q^{\tau n}).
$$
This implies what we have to prove with $\alpha =\tau$.

\end{proof}

\section{Sequences realized by dynamical systems}
Let $(\Omega, \mathcal{A}, \nu, S)$ be a measure-preserving dynamical system. For a given function $h \in L^1(\nu)$, the sequence $(h(S^n \omega))$ is referred to as a random
sequence realized by the dynamical system $(\Omega, \mathcal{A}, \nu, S)$.  Under some condition on $h$, the sequence  $(h(S^n \omega))$ is almost surely of Gelfond type with
Gelfond exponent $\frac{1}{2}$ and is fully oscillating.  More precisely we have the following result. 

\begin{thm} \label{D-Gelfond} Let $(\Omega, \mathcal{A}, \nu, S)$ be a measure-preserving dynamical system. Suppose $h \in L^\infty(\nu)$ with $\int h d\nu=0$ and
\begin{equation}\label{HNcondition}
    \sum_{n=1}^\infty \| \mathbb{E}(h | S^{-n} \mathcal{A})\|_\infty<\infty.
\end{equation}  
Then for $\nu$-almost every $\omega$ and for any $d\ge 1$ we have
\begin{equation}\label{R-G}
\sup_{P\in \mathbb{R}_d[t]} \left|\sum_{n=0}^{N-1} h(S^n\omega)e^{2\pi i P(n)} \right| = O_{d,\omega}\left(\sqrt{N \log N} \right).
\end{equation}
\end{thm}  

\begin{proof}
We can repeat the proof of Theorem 2 in \cite{F}. But the key point that we should pointed out for the situation considered here  is that under the condition (\ref{HNcondition}), the sequence $(h(T^n \omega))$
shares the following subgaussian property: for any $(a_n) \in \ell^2(\mathbb{N})$ we have 
\begin{equation}\label{subg} 
          \mathbb{E} \exp\left(\lambda \left|\sum_{n=1}^\infty a_n h(S^n \omega)\right|\right) \le \exp \left(c \lambda^2 \sum_{n=1}^\infty |a_n|^2\right)
\end{equation}
where $c>0$ is a constant  depending on $h$. This inequality (\ref{subg}) was proved in \cite{F0} (see Lemma 2.5 and the proof of Theorem 2.6 in \cite{F0}).
\end{proof}

Recall that $\mathbb{E }(h | S^{-n} \mathcal{A}) = L^n h$ where $L$ is the Perron-Frobenius operator of the dynamical system which is defined by 
$$
        \int  f L h d\nu = \int f\circ S h d\nu \quad  (\forall f\in L^1(\nu), \forall h \in L^\infty(\nu).
$$
The decay of $\|L^n h\|_\infty$ was well studied and we know that the condition (\ref{HNcondition})  is satisfied by many dynamical systems and many
"regular" functions $h$.      	See for example, \cite{Baladi} and \cite{FJ0}.

The subgaussian property (\ref{subg}) was proved for more general random sequences in \cite{CF}. So these  random sequences also share the same property
stated  in Theorem \ref{D-Gelfond}.

Let us mention the following results due to Lesigne \cite{L1993} describing conditions for $h(S^n \omega)$ to be almost surely oscillating of a given order.
The conditions are expressed by the quasi-discrete spectrum  of the system.  Let  $(\Omega, \mathcal{A}, \nu, S)$ be an ergodic  measure-preserving dynamical system. The isometry
$g \mapsto g\circ S$ on $L^2(\nu)$ is still denoted by $S$. Let $E_0(S)$ be the set of  eigenvalues of $S$. For any $k\ge 1$, we inductively define
$$
    E_k(S) := \{g \in L^2(\nu): |g|=1 \ \ \nu\!\!-\!\!a.e., g\circ S \cdot \overline{g} \in E_{k-1}(S)\}.
$$
The set $E_1(S)$ is essentially the set of eigenfunctions of $S$ (up to multiplicative constants). Functions in $E_k(S)$ are the $k$-th generalized eigenfunctions of $S$. 
The sequence $E_k(S)$ is decreasing. Let 
$$
     E_\infty(S) = \bigcup_{k\ge 0} E_k(S).
$$
For $h\in L^1(\nu)$, by $f \in E_k(S)^\perp$ we mean $\int h gd\nu =0$ for all $g\in E_k(S)$.
The following results are proved by Lesigne in \cite{L1993}. We state them in terms of the notion of oscillation :
\begin{itemize}
\item[(i)] Assume that $(\Omega, \mathcal{A}, \nu, S)$ is  totally ergodic and $h\in L^1(\nu)$. The sequence $(h(S^n \omega))$ is almost surely oscillating of order $d$
if and only if $f \in E_d(S)^\perp$.
\item[(ii)] Assume that $(\Omega, \mathcal{A}, \nu, S)$ is  weakly mixing (i.e. $1$ is the only eigenvalue). For any $h\in L^1(\nu)$ with $\int h d\nu=0$, the sequence $(h(S^n \omega))$ is almost surely fully oscillating. Moreover
almost surely for any $d\ge 1$,
\begin{equation}\label{L-Osc}
   \lim_{N\to \infty}\sup_{P\in \mathbb{R}_d[t]}\left| \frac{1}{N} \sum_{n=1}^N h(S^n \omega) e^{2\pi i P(n)}\right|=0.
\end{equation}

\end{itemize}
%As pointed out by Lemancyzk (private communication),  the above results are expressed using the notion of oscillation.

 The weak mixing condition means  $E_1(S) = U$, the set of complex numbers of modulus $1$. In other words, $E_\infty(S) = U$.
 In this case, the condition $\int h d\nu =0$ means $h \in E_\infty(S)^\perp$. 
 
 According to Abramov \cite{Abramov}, we say that $(\Omega, \mathcal{A}, \nu, S)$ has {\em quasi-discrete spectrum} if $E_\infty(S)^\perp = \{0\}$. From (i) we deduce
 \begin{itemize}
\item[(iii)] If $(\Omega, \mathcal{A}, \nu, S)$ is  totally ergodic and has quasi discrete spectrum, then there is no  $h\in L^1(\nu)$ with $h\not=0$ such that  $(h(S^n \omega))$ is almost surely fully oscillating.
\item[(iv)] Assume that $(\Omega, \mathcal{A}, \nu, S)$  is  totally ergodic and doesn't have quasi discrete spectrum.  For any $h\in E_\infty(S)^\perp $, the sequence $(h(S^n \omega))$ is almost surely fully oscillating. 
\item [(v)] There are nilsequences which are fully oscillating. In fact,  let us consider a Heisenberg  nilrotation $(X, S)$. Take a  function $h\in L^2$ which is orthogonal to the quasi-discrete part, then  for a.e. $x$
 the nilsequence $h(S^n x)$ is fully oscillating.  
\end{itemize}

Furthermore, we can take a continuous  function $h$ in (v) which is orthogonal to the quasi-discrete part. Then  for every  $x$
 the nilsequence $h(S^n x)$ is fully oscillating, by the topological polynomial Wiener-Wintner theorem proved in \cite{F3}.
%These remarks go out of communications with  Lemancyzk. 

The property (\ref{L-Osc}) is to be compared with (\ref{R-G}). The weaker property (\ref{L-Osc}) holds for all $h\in L^1(\nu)$ with $\int h d\nu=0$.  and the property (\ref{R-G})
is stronger but requires more regularity on $h$.

\medskip

\noindent {\em Acknowledgement.}  I would like to thank M. Lemanczyk and Ch. Mauduit  for valuable discussions and useful references.

% TO DO ! Reference for self reading
%\bibliographystyle{plain}
%%    Insert the bibliography data here.
%\bibliography{q_multiplicative}

\end{document}